\documentclass[reqno,12pt]{amsart}

\usepackage[dvipsnames]{xcolor}
\usepackage{microtype}

\usepackage{tikz}
\usepackage{tikz-cd}
\usepackage[all]{xy}
\usepackage[utf8]{inputenc}  
\usepackage[english]{babel}
\usepackage[T1]{fontenc}
\usepackage{lmodern}
\usepackage{multicol}
\usepackage{color}
\usepackage{amsmath, amsfonts, amssymb, amsthm}
\usepackage{graphicx}
\usepackage{enumerate}
\usepackage{caption}
\usepackage{geometry}
\usepackage{enumitem}
\usepackage{multirow}
\usepackage{stmaryrd}
\usepackage{url}

\usepackage{bbm}
\usepackage{subcaption}
\usepackage{mathtools}
\usepackage{lipsum}
\usepackage[title,titletoc]{appendix}
\usepackage{booktabs}
\usepackage{here}
\usepackage[colorlinks=true, pdfstartview=FitV, linkcolor=purple, citecolor=blue, filecolor=violet, urlcolor=violet]{hyperref}

\newcommand{\Z}{\mathbb{Z}}

\newcommand{\D}{\mathbb{D}}

\newcommand{\rep}{\operatorname{rep}}
\newcommand{\md}{\operatorname{mod}}
\newcommand{\ind}{\operatorname{Ind}}
\newcommand{\tagg}{\operatorname{tag}}
\newcommand{\T}{\mathcal{T}}
\newcommand{\arrow}{i \rightarrow j}
\newcommand\norm[1]{\left\lvert#1\right\rvert}

\theoremstyle{definition}
\newtheorem{theorem}{Theorem}[section]
\newtheorem{cor}[theorem]{Corollary}

\newtheorem{prop}[theorem]{Proposition}

\newtheorem*{theorem*}{Theorem}

\theoremstyle{definition}
\newtheorem{definition}[theorem]{Definition}

\newtheorem{example}[theorem]{Example}

\theoremstyle{remark}
\newtheorem{remark}[theorem]{Remark}

\title[Geometric model for the cluster category of affine type D]{A geometric model for the non-homogeneous tubes of the cluster category of affine type D}
\author{Amandine Favre}
\date{\today}

\address{Faculty of Mathematics, Ruhr-Universit\"at Bochum, Universit\"atsstr. 150, 44801 Bochum, Germany}
\email{amandine.favre@ruhr-uni-bochum.de}

\begin{document}

\begin{abstract}
In this article, we give a geometric model for non-homogeneous tubes of the cluster category of the affine type $D$. This model is given in terms of homotopy classes of unoriented arcs in the twice punctured disk. 
In particular, we extend the geometric model for the tube of rank $n-2$ given in \cite{TypeD} to the two tubes of rank $2$.
\end{abstract}

\maketitle 

\setcounter{tocdepth}{2} 
\tableofcontents

\section{Introduction}
Cluster algebras were introduced in 2002 by Fomin and Zelevinsky in \cite{FZ02}. A cluster algebra is a commutative algebra with generators called cluster variables.
A large class of cluster algebras can be defined from tagged triangulations \cite{FST} where the cluster variables correspond to tagged arcs in the corresponding surface. For example, a cluster algebra of affine type $D$ comes from a tagged triangulation of a twice-punctured disk.

Around 2006, cluster categories were introduced by Buan, Marsh, Reineke, Reiten and Todorov in \cite{BMR+}, and independently by Caldero, Chapoton and Schiffler in \cite{CCS}.
A link were established between cluster algebras and representations of quiver. Objects in a cluster category correspond to cluster variables under the cluster character map also called Caldero–Chapoton map (CC-map).

Various works have been done on the categories of clusters associated with surface triangulations. For example, \cite{CCS} modelled the cluster category of type $A_n$ by a disk with $n+3$ marked points on the boundary or \cite{Schiffler_D} modelled the cluster category of type $D_n$ by a disk with a puncture and $n$ marked points on the boundary.
The cluster category of affine type $D$ arises from a triangulation of a twice-punctured disk and objects correspond to tagged arcs in this surface \cite{BQ,QZ,AP}.

Our view here is that by using arcs, we can model the non-homogeneous tubes of the cluster category of affine type $D$. Similar methods are used by \cite{BM_An} to model the tube categories or by \cite{BZ} to model the cluster category of a surface without punctures.
We note that usually geometric models of cluster categories use unoriented arcs (see \cite{BZ,CCS}) and geometric models of module categories use oriented arcs (see \cite{BM_An}). In this paper, we use unoriented arcs and focus on the cluster category. 

In this article, we complete the geometric model described for example in \cite[Proposition 3.5]{TypeD} for the tube of rank $n-2$ of the cluster category of affine Type $D$ by giving a new model for the two tubes of rank $2$. 
In order to do that we introduce new arcs between the punctures of the twice-punctured disk,calling them interior arcs, and construct a quiver on them. This quiver has two connected components and is isomorphic to the tubes of rank $2$ in the Auslander-Reiten quiver of the cluster category. 


In \cite{jackson}, the author gives another geometric model for the module category of affine type $D$ introducing a new parameter on the arcs, the colour. 
The use of coloured arcs arises to a more complex model to be able to describe all the tubes. 
In this paper, we avoid the use of coloured arcs by describing only the non-homogeneous tubes.

In Section \ref{sec:triang}, we recall notions and notations about triangulations of marked surfaces and their associated quiver, that we will use throughout the article. In Section \ref{sec:categories} we recall some important facts about module and cluster categories of affine type $D$. In Section \ref{sec:n-2}, we study the arcs in the twice-punctured disk and we recall the geometric model for the tube of rank $n-2$ of the cluster category of affine type $D$ given in \cite{TypeD}. In Section \ref{sec:2}, we extend the definition of arcs in order to extend the geometric model to the two tubes of rank $2$.

\section*{Acknowledgements}
I thank my supervisor Karin Baur for suggesting this topic and introducing me to this subject. I thank her for the helpful conversations and her careful reading of my work.

\section{Tagged triangulations and associated quivers}\label{sec:triang}
\subsection{Tagged arcs and triangulations}

In this section, we define the notions of tagged arcs and triangulations, following \cite{FST}.
Let $S$ be a connected oriented Riemann surface with boundary. Let $M$ be a finite set of marked point on the boundary and in the interior of $S$. The marked points in the interior of $S$ are called the \textit{punctures}, the set of all punctures is denoted by $P$ and the pair $(S,M)$ is called a \textit{marked surface}. 
We will first introduce notations for arbitrary surfaces but then $(S,M)$ will be a twice-punctured disk with at least two marked points on the boundary. 

\begin{definition} An \textit{arc} $\gamma$ in $(S,M)$ is a curve in $S$, up to isotopy fixing the endpoints, such that
\begin{itemize}
\item the endpoints of $\gamma$ are marked points in $M$;
\item $\gamma$ does not cross itself (except that its endpoints may coincide);
\item except for its endpoints, $\gamma$ is disjoint from $\partial S$ and $M$; 
\item $\gamma$ does not cut out an unpunctured monogon or an unpunctured digon.
\end{itemize}
\end{definition}
Two arcs are called \textit{compatible} if they do not intersect in the interior of $S$. 

\begin{definition}
A \textit{triangulation} of $(S,M)$ is a maximal collection of pairwise compatible distinct arcs in $(S,M)$. The arcs of a triangulation divide the surface into \textit{triangles}. We label the arcs in a triangulation $T$ by $\{1,...,n\}$. 

A triangle with only two distinct side is called a \textit{self-folded triangle}. The arc inside a self-folded triangle is called its \textit{radius} and the arc that encircles the puncture is called its \textit{loop}.
\end{definition}

To avoid self-folded triangle, we can use tagged arcs and tagged triangulations as in \cite[Section 7]{FST}. For each self-folded triangle, we replace the loop by a tagged arc parallel to the radius, with a "tag" at the puncture. An example of the use of tagged arcs is shown in Figure \ref{Ex:Tagged}.

\begin{definition}
Let $\gamma$ be an arc in $(S,M)$. This arc has two ends. A \textit{tagged arc} is an arc in which each end has been tagged as \textit{unnotched} (plain) or \textit{notched}, so that the following conditions hold: 
\begin{itemize}
\item the arc does not cut out a once punctured monogon;
\item any endpoint on the boundary is tagged unnotched;
\item both ends of a loop are tagged the same way.
\end{itemize}
Two tagged arcs $\gamma_1$, $\gamma_2$ are called \textit{compatible} if all of the following hold: 
\begin{itemize}
\item their untagged versions are compatible;
\item if their untagged version are different and $\gamma_1$ and $\gamma_2$ share an endpoint, then they have the same tag at this endpoint;
\item if their untagged version coincide, then they have the same tag at least at one end.
\end{itemize}
\end{definition}

\begin{definition}
A \textit{tagged triangulation} of $(S,M)$ is a maximal collection of pairwise compatible distinct tagged arcs in $(S,M)$.
\end{definition}

\begin{figure}[h]
\includegraphics[scale=.6]{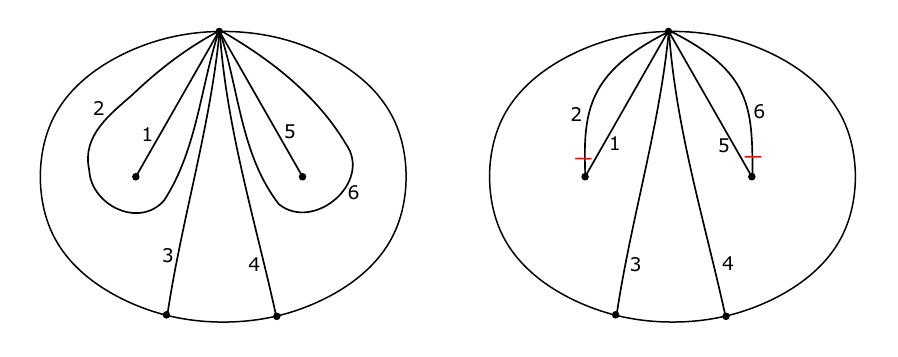} 
\caption{Example of a triangulation and its associated tagged triangulation}\label{Ex:Tagged}
\end{figure}

\begin{remark}
\cite{QZ} introduce arcs as parametrized curves, i.e. as functions $\gamma:[0,1]\rightarrow S$, with $\gamma(0)$, $\gamma(1)\in M$ and $\gamma(t)\in S\setminus M$ for $0<t<1$ up to homotopy fixing the endpoints, and with a tag function defined as $\kappa:\{t \mid \gamma(t)\in P\}\rightarrow \{0,1\}$, where $P\subset M$ is the set of punctures.
\end{remark}

\subsection{Quivers from triangulations}
From a triangulation $T$, we can construct a quiver $Q_T$ corresponding to $T$. It is called the \textit{quiver} of $T$ and is constructed as follows:
\begin{enumerate}
\item for each arc $i$ in $T$ we draw an associated vertex $i$ in $Q_T$;
\item in each non self-folded triangle in $T$, we draw an arrow for each angle of the triangle as follow: let $i$ and $j$ be the two arcs that compose this angle, we draw $\arrow$ if $i$ is followed by $j$ clockwise and $j\rightarrow i$ if $j$ is followed by $i$ clockwise
;
\item in each self-folded triangle in $T$ with radius $i$ and loop $l$, for every $k\in Q_0$ with an arrow $l\rightarrow k$ we draw an arrow $i \rightarrow k$ and for for every $k\in Q_0$ with an arrow $k\rightarrow l$ we draw an arrow $k \rightarrow i$; 
\item we remove all $2$-cycles.
\end{enumerate}

The quiver associated to the triangulation on the left in Figure \ref{Ex:Tagged} is the following: 
\[\begin{tikzcd}[column sep=scriptsize,row sep=tiny]
	1 &&& 5 \\
	& 3 & 4 \\
	2 &&& 6
	\arrow[from=1-1, to=2-2]
	\arrow[from=2-2, to=2-3]
	\arrow[from=2-3, to=1-4]
	\arrow[from=2-3, to=3-4]
	\arrow[from=3-1, to=2-2]
\end{tikzcd}\]


From now on, $(S,M)$ will be a twice-punctured disk with at least two marked points on the boundary.

\section{Cluster and module categories} \label{sec:categories}
Let $k$ be an algebraically closed field and consider an algebra $\Lambda$ of affine type $D$. 
Recall that, for $Q$ a quiver, and $X$ a finite or affine Dynkin diagram, then we say that $Q$ is of type $X$ if $Q$ is mutation equivalent to an acyclic orientation of $X$, for details about quiver mutations, see \cite[Section 8]{FZ2}. By extension, we say that an algebra is of type $X$ if it can be described as the path algebra $kQ$ of a quiver $Q$ of type $X$.
So we can assume that $\Lambda$ is the path algebra $kQ$, where $Q$ is any orientation of the following quiver with $n+1$ vertices:
\[\begin{tikzcd}[sep=tiny]
	1 &&&&&& n \\
	& 3 & 4 & 5 & {(n-2)} & {(n-1)} \\
	2 &&&&&& {(n+1)}
	\arrow[no head, from=1-1, to=2-2]
	\arrow[no head, from=2-2, to=2-3]
	\arrow[no head, from=2-3, to=2-4]
	\arrow[dotted, no head, from=2-4, to=2-5]
	\arrow[no head, from=2-5, to=2-6]
	\arrow[no head, from=2-6, to=1-7]
	\arrow[no head, from=2-6, to=3-7]
	\arrow[no head, from=3-1, to=2-2]
\end{tikzcd}\]


The surface associated to the cluster category of an algebra of affine type $D_n$ is the twice-punctured disk with $n-2$ marked points on the boundary $\D(n)$, see for example \cite[Example 6.10]{FST}. We will label the $n-2$ marked points on the boundary as $\{1,2,...,n-2\}$ counterclockwise, and the two punctures as $P$ and $Q$ as shown in Figure \ref{fig:marked_pts}.

\begin{figure}[h]
\includegraphics[scale=0.7]{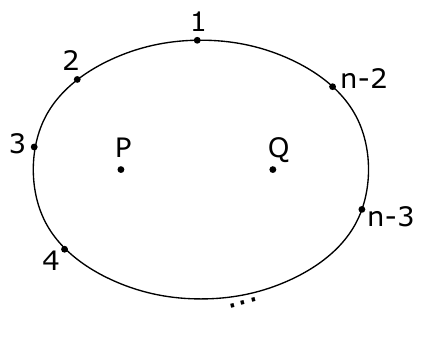} 
\caption{Labelling of the marked points of $\D(n)$} \label{fig:marked_pts}
\end{figure}

\begin{remark}
In this paper, we will use both the term representation of $Q$ and module of $kQ$, depending on the context, since in this case we have an equivalence of categories between the category of finite dimensional $k$-representation of $Q$, $\rep Q$ and the category of modules over the path algebra $kQ$, $\md kQ$ (\cite[Cor. III.1.7]{ASS}).
\end{remark}

\subsection{Cluster category of affine type D} For an algebra $\Lambda$ of affine type $D$, its module category is of infinite type. The module category of $\Lambda$ is Krull-Schmidt and so it can be described by the Auslander-Reiten quiver of the category.
The vertices of the Auslander-Reiten quiver are the isomorphism classes of indecomposable modules, and the arrows are the irreducible morphisms between them. 


The Auslander-Reiten quiver of the module category of an algebra of affine type D is describe as follow: 
\[\Gamma(\md \Lambda)= \mathcal{P}(\Lambda)\cup \mathcal{R}(\Lambda)\cup \mathcal{Q}(\Lambda) \]
where $\mathcal{P}(\Lambda)$ is the postprojective component, $\mathcal{Q}(\Lambda)$ is the preinjective component and $\mathcal{R}(\Lambda)$ is the regular component. 
The regular component $\mathcal{R}(\Lambda)$ consists of a $\mathbb{P}_1(k)$-family of stable tubes, with one tube of rank $n-2$, two tubes of rank $2$ and the remaining ones of rank $1$. These tubes are illustrated in Figure \ref{fig:tubes}.
We will denote by $\T_1$ the tube of rank $n-2$ and by $\T_2$ and $\T_3$ the two tubes of rank $2$. Any module at the mouth of a tube is called a \textit{quasi-simple} module. For more details, see \cite[Chapter XIII]{SS2}. 

The associated cluster category has the same family of tubes, together with the so called transjective component, see \cite{BMR+} for more details. 

\begin{figure}[h]
\includegraphics[scale=0.4]{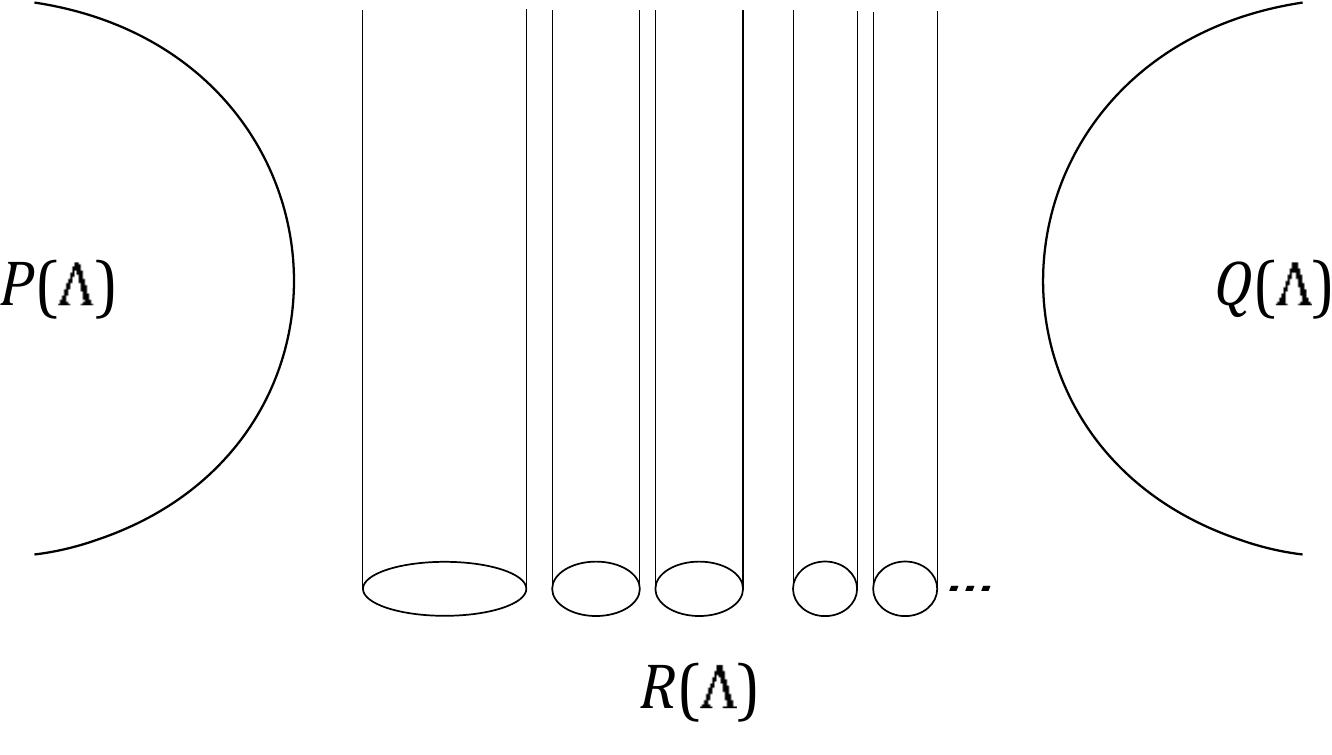} 
\caption{Components of the module category schematically} \label{fig:tubes}
\end{figure}

In Section \ref{sec:n-2}, we will recall a geometric model for the tube $\T_1$ of rank $n-2$, following \cite{TypeD}, and in Section \ref{sec:2} we will extend this model for the two tubes $\T_2$ and $\T_3$ of rank $2$. 


\section{Tube of rank $n-2$ in the cluster category of affine type $D$} \label{sec:n-2}
In this section, we will first describe the arcs in the twice punctured disk. Then we will recall the bijection between the peripheral generalized arcs and the indecomposable modules in the tube of rank $n-2$ of the cluster category of affine type $D_n$, following \cite[Section 3.2]{TypeD}. 

\subsection{Arcs in the twice-punctured disk} We will distinguish four types of arcs in the twice-punctured disk and describe the effect of the Auslander-Reiten translate on the representations associated to the arcs of each types, following \cite[Section 3.2]{TypeD}. 

\begin{definition}
A \textit{generalized arc} is an arc which is allowed to cross itself. 
\end{definition}

\begin{definition}
A \textit{peripheral arc} is a generalized arc $\gamma$ with both endpoints on the boundary, such that $\gamma$ is homotopic to a concatenation of at least two boundary segments. We consider that the arc follow the boundary counterclockwise, so $\gamma$ is defined by its starting marked point $s\in\{1,2,...,n-2\}$, the number of full turns around the boundary $l\in \Z_{\geq 0}$ and its ending marked point $t\in\{1,2,...,n-2\}$. We will denote it by $\gamma^l_{s,t}$.
\end{definition}

In the twice-punctured disk $\D(n)$, the four different types of generalized arcs we consider, are shown in the Figure \ref{fig:type_arcs}.
\begin{enumerate}[label=(\roman*)]
\item peripheral arcs;
\item arcs $\gamma$ with both endpoints on the boundary, such that $\gamma$ is not homotopic to a concatenation of boundary segments; 
\item arcs with one endpoint on the boundary and the other at a puncture; 
\item arcs with both endpoints at punctures.
\end{enumerate}

\begin{figure}[h]
\includegraphics[scale=0.7]{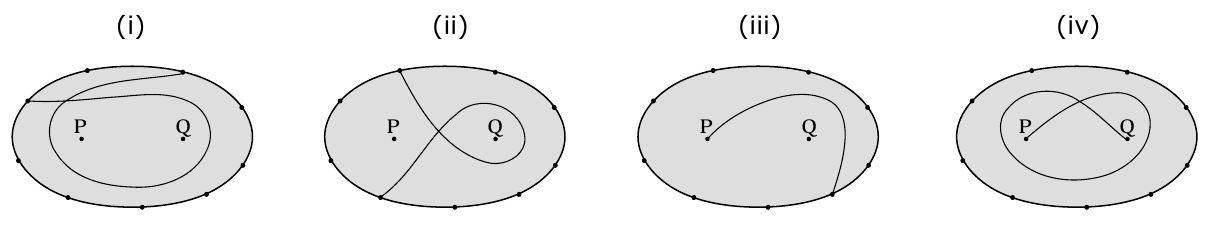} 
\caption{Four types of arcs} \label{fig:type_arcs}
\end{figure}


Brüstle and Zhang describe the effect of the Auslander-Reiten translate on generalized arcs on unpunctured surfaces in \cite{BZ}. For the peripheral arcs in affine type $D$, we can lift the setup to the affine type $A$ to use the results from \cite{BM_An}.
Brüstle and Qui describe the effect of the Auslander-Reiten translate on arcs on punctured surfaces in \cite{BQ}. 

\begin{prop}\label{prop:tau}
\begin{enumerate}
\item {\cite[Proposition 1.3]{BZ}} Let $\gamma$ be a generalized arc in $(S,M)$ of type $(i)$ or $(ii)$, with both endpoints on the boundary. We denote by $M(\gamma)$ the corresponding indecomposable object in the cluster category. Let $\tau$ be the Auslander–Reiten translate. Then the arc associated to $\tau(M(\gamma))$ is obtained by moving both endpoints of $\gamma$ clockwise to the next marked points on the boundary. In particular, if $\gamma=\gamma^l_{s,t}$ is a peripheral arc, we have 
\[\tau(M(\gamma^l_{s,t}))=M(\gamma^l_{s-1,t-1}).\]
\item {\cite[Section 3.2]{BQ}} Let $\gamma$ be an arc with one endpoint at a puncture and the other on the boundary. Then the arc associated to $\tau(M(\gamma))$ is obtained by changing the tag at the puncture and by moving the other endpoint clockwise to the next marked point on the boundary. \label{prop:tau2}
\item {\cite[Section 3.2]{BQ}} Let $\gamma$ be an arc with both endpoints at puncture. Then the arc associated to $\tau(M(\gamma))$ is obtained by changing the tags at the punctures.  \label{prop:end_punct}
\end{enumerate}
\end{prop}

\begin{remark}
Note that \cite{BQ} does not consider generalized arcs between punctures.
\end{remark}

From this Proposition, we can make the following observation about the periodicity of the Auslander-Reiten translate on the curves:

\begin{cor}[{\cite[Corollary 3.4]{TypeD}}]
Let $M$ be an indecomposable object, such that there exists $\gamma$ with $M(\gamma)=M$. If $M$ is $\tau$-periodic, then $\gamma$ cannot be of type $(ii)$ or of type $(iii)$.
\end{cor}

We also observe that for $\gamma\in (S,M)$ a peripheral arc, we have that $M(\gamma)$ is $(n-2)$-periodic, and for $\delta\in (S,M)$ an arc (not generalized) with both endpoints at punctures, $M(\delta)$ is $2$-periodic. 

\subsection{Tube $\T_1$ of rank $n-2$} We recall the bijection between the peripheral generalized arcs and the indecomposable modules in the tube of rank $n-2$ of the cluster category of affine type $D$, following \cite[Proposition 3.5]{TypeD}.

\begin{prop} \label{Prop:T1}
Let $\ind(\T_1)$ be the set of indecomposable modules in the tube $\T_1$, up to isomorphism. Then we have the following bijection:
\[ \ind(\T_1) \longleftrightarrow \{\text{peripheral generalized arcs }\gamma^l_{s,t}\}, \]
so the indecomposable modules are in bijection with the set of peripheral arcs $\gamma^l_{s,t}$ with $s,t\in\{1,2,...,n-2\}$, $l\in \Z_{\geq 0}$ and $t\neq s,s+1$ if $l=0$. 
In addition, we have that the irreducible maps in the tube correspond to lengthening and shortening the generalized arcs of a marked point on the boundary and that the quasi-simple modules of $\T_1$ correspond to the peripheral arcs $\gamma^0_{s,s+2}$, with $s\in\{1,2,...,n-2\}$, reduced modulo $n-2$. 
\end{prop} 

\begin{example}
In this example, we consider an algebra of affine type $D_7$. A part of the Auslander-Reiten quiver of the tube $\T_1$ in term of arcs is shown in the Figure \ref{fig:T1}.
\end{example}

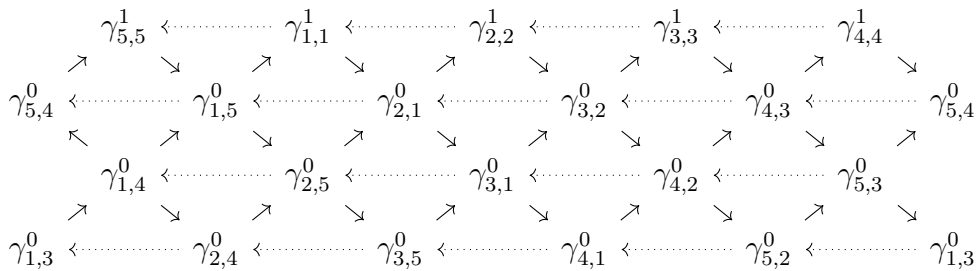
\begin{figure}[h]
\[\begin{tikzcd}[sep=tiny]
	& {\gamma^1_{5,5}} && {\gamma^1_{1,1}} && {\gamma^1_{2,2}} && {\gamma^1_{3,3}} && {\gamma^1_{4,4}} \\
	{\gamma^0_{5,4}} && {\gamma^0_{1,5}} && {\gamma^0_{2,1}} && {\gamma^0_{3,2}} && {\gamma^0_{4,3}} && {\gamma^0_{5,4}} \\
	& {\gamma^0_{1,4}} && {\gamma^0_{2,5}} && {\gamma^0_{3,1}} && {\gamma^0_{4,2}} && {\gamma^0_{5,3}} \\
	{\gamma^0_{1,3}} && {\gamma^0_{2,4}} && {\gamma^0_{3,5}} && {\gamma^0_{4,1}} && {\gamma^0_{5,2}} && {\gamma^0_{1,3}}
	\arrow[from=1-2, to=2-3]
	\arrow[dotted, from=1-4, to=1-2]
	\arrow[from=1-4, to=2-5]
	\arrow[dotted, from=1-6, to=1-4]
	\arrow[from=1-6, to=2-7]
	\arrow[dotted, from=1-8, to=1-6]
	\arrow[from=1-8, to=2-9]
	\arrow[dotted, from=1-10, to=1-8]
	\arrow[from=1-10, to=2-11]
	\arrow[from=2-1, to=1-2]
	\arrow[from=2-3, to=1-4]
	\arrow[dotted, from=2-3, to=2-1]
	\arrow[from=2-3, to=3-4]
	\arrow[from=2-5, to=1-6]
	\arrow[dotted, from=2-5, to=2-3]
	\arrow[from=2-5, to=3-6]
	\arrow[from=2-7, to=1-8]
	\arrow[dotted, from=2-7, to=2-5]
	\arrow[from=2-7, to=3-8]
	\arrow[from=2-9, to=1-10]
	\arrow[dotted, from=2-9, to=2-7]
	\arrow[from=2-9, to=3-10]
	\arrow[dotted, from=2-11, to=2-9]
	\arrow[from=3-2, to=2-1]
	\arrow[from=3-2, to=2-3]
	\arrow[from=3-2, to=4-3]
	\arrow[from=3-4, to=2-5]
	\arrow[dotted, from=3-4, to=3-2]
	\arrow[from=3-4, to=4-5]
	\arrow[from=3-6, to=2-7]
	\arrow[dotted, from=3-6, to=3-4]
	\arrow[from=3-6, to=4-7]
	\arrow[from=3-8, to=2-9]
	\arrow[dotted, from=3-8, to=3-6]
	\arrow[from=3-8, to=4-9]
	\arrow[from=3-10, to=2-11]
	\arrow[dotted, from=3-10, to=3-8]
	\arrow[from=3-10, to=4-11]
	\arrow[from=4-1, to=3-2]
	\arrow[from=4-3, to=3-4]
	\arrow[dotted, from=4-3, to=4-1]
	\arrow[from=4-5, to=3-6]
	\arrow[dotted, from=4-5, to=4-3]
	\arrow[from=4-7, to=3-8]
	\arrow[dotted, from=4-7, to=4-5]
	\arrow[from=4-9, to=3-10]
	\arrow[dotted, from=4-9, to=4-7]
	\arrow[dotted, from=4-11, to=4-9]
\end{tikzcd}\]
\caption{Part of the Auslander-Reiten quiver of the tube $\T_1$}\label{fig:T1}
\end{figure}

\begin{remark}
The objects in the Auslander-Reiten quiver depend on the triangulation $T$. The modules corresponding to the arcs in $\D_n$ are given by the intersection between the triangulation and the arcs shown in the previous example.
\end{remark}

\section{Tubes of rank $2$ in the cluster category of affine type $D$} \label{sec:2} 
In this section, we will give a new description of the two tubes $\T_2$ and $\T_3$ of rank $2$ in terms of generalized arcs. For this purpose, we will study the generalized arcs in the twice-punctured disk $\D(n)$, define a new type of arcs to then determine the arcs which correspond to the objects in these two tubes. 

\subsection{Arcs in the tubes of rank $2$}
We first want to find four possible arcs to correspond to the quasi-simple modules in the tubes $\T_2$ and $\T_3$. 
We recall that the quasi-simple modules are the modules at the mouth of the tubes $\T_2$ and $\T_3$. So there are only four quasi-simple modules in the tubes $\T_2$ and $\T_3$ and these four objects have to be $\tau$-periodic of period $2$ since the two tubes are of rank $2$.
Moreover, these quasi-simple modules are the only rigid modules in these tubes and the rigid modules correspond to arcs without self-crossings, \cite[Theorem 5.5]{QZ}.
Since the peripheral arcs correspond to the objects in the tube $\T_1$ by Proposition \ref{Prop:T1}, the Proposition \ref{prop:tau}(\ref{prop:end_punct}) tells us that the arcs corresponding to the objects in the tubes $\T_2$ and $\T_3$ have both endpoints at the punctures. 
The corresponding arcs are those shown in Figure \ref{fig:PQ_arcs} (cf \cite[Lemma 3.7]{TypeD}). We use these four arcs to construct a geometric model for the two tubes.

\begin{figure}[h]
\includegraphics[scale=0.7]{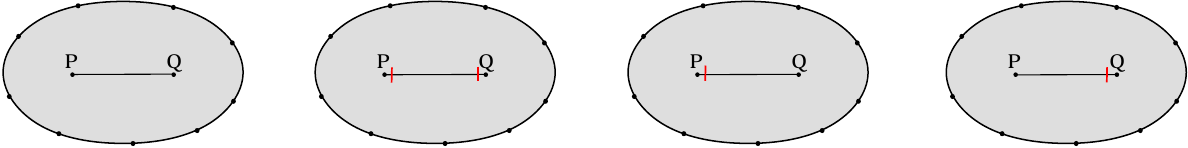}
\caption{Four arcs at the mouth of the tube} \label{fig:PQ_arcs}
\end{figure}

To describes the two tubes, we need to extend the definition of generalised arcs, so we introduce a new class of arcs, which are allowed to cut out one-punctured monogon.
\begin{definition}
An \textit{interior arc} $\gamma$ in $(S,M)$ is a generalized tagged arc allowed to cut out a one-punctured monogon, such that \begin{itemize}
\item both endpoints of $\gamma$ are punctures;
\item if $\gamma$ cut out a one-punctured monogon, then its ends are tagged differently. 
\end{itemize}
\end{definition}

For all interior arcs that are not homotopic to those in Figure \ref{fig:PQ_arcs}, we will view them as winding counterclockwise around the two punctures, associating an orientation to them. 

We introduce a notation for interior arcs: 
Consider a (bounded) cylinder, with one marked point $P$ on the lower boundary and one marked point $Q$ on the upper boundary. We look at this cylinder as a rectangle with the vertical sides identified, and we work in its universal covering space, an infinite strip in the plane. 
It will be convenient to denote the preimages of $P$ by odd number on the lower boundary and the preimages of $Q$ by even number on the upper boundary. 

\begin{figure}[h]
\includegraphics[scale=0.45]{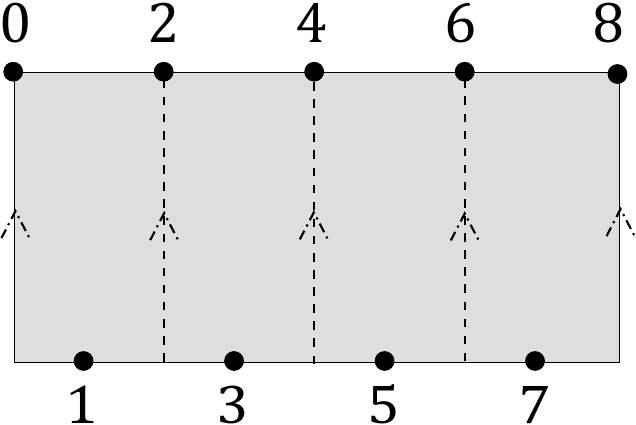} 
\caption{A part of the universal covering space}
\end{figure}

We will use the arcs in this cylinder to construct the arcs in the twice punctured disk with both endpoints at punctures. 
We will denote arcs in the cylinder as $[x,y]$ with $x,y\in\Z\sqcup \Z^\ast$, writing $\ast$ to denote a notched tag at this end.
The arc $\gamma=[x,y]$ in the cylinder connects $x$ and $y$.
Let $n$ be the number of full turns around the cylinder done by $\gamma=[x,y]$. This number of winding $n$ is also the number of times $\gamma$ crosses a dotted line in the universal cover, except the intersections at the endpoints. We observe that $n=\left\lfloor\dfrac{\norm{y-x}}{2}\right\rfloor$.

The arc $\gamma$ in the cylinder will correspond to the arc in the twice-punctured disk that begins at the puncture corresponding to the marked point on the left end of $\gamma$ in the cylinder, and will wind around the punctures following the boundary $n$ times and then ends at the puncture corresponding to the marked point on the right end of $\gamma$ in the cylinder.

\begin{figure}[h]
\includegraphics[scale=0.7]{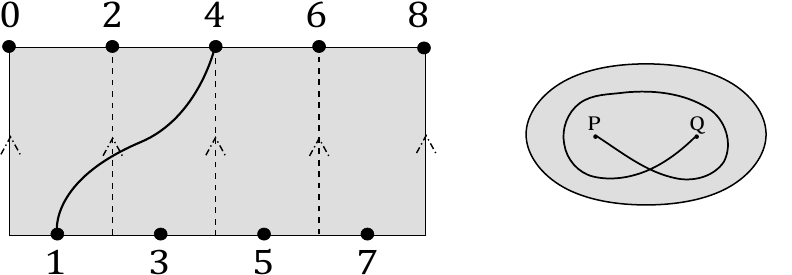} 
\caption{A curve in the cylinder and its corresponding curve in the twice punctured disk}
\end{figure}

For $a\in \Z\sqcup \Z^\ast$, we write $a^\ast$ to indicate the change of tagging at this end of the curve, $\tagg(a)$ for the tag (plain or notched) at $a$ and $\norm{a}$ for the integer $a\in \Z$ regardless to the tagging.
Let $x,y \in\Z\sqcup \Z^\ast$, we observe that:
\begin{enumerate}[label=\arabic*.]
\item the arc $[x,y]$ is the same arc as $[y,x]$. In other term, we only need to know the starting point and the ending point, since the arcs are read from left to right. We will only consider arcs $[x,y]$ with $\norm{x}\leq \norm{y}$;
\item by construction, $[x,y]=[x-2m,y-2m]$ for all $m\in\Z$.
Moreover, the arcs $[x,x]$ and $[x^\ast,x]$ are contractible, so we will only consider arcs $[x,y]$ with $\norm{x}\neq \norm{y}$;
\item let $x<y$ and $a<b$ in $\Z\sqcup \Z^\ast$. Then the arc $[x,y]$, $\norm{x}\not\equiv \norm{y} \mod 2$, i.e., connecting $P$ and $Q$, is equal to the arc $[a,b]$ if and only if $\norm{y}-\norm{x}=\norm{b}-\norm{a}$ and their tags agree by parity.
\end{enumerate}

\begin{example}
The first cylinder of Figure \ref{examples_arcs} illustrates the second observation above. The blue arc is denoted $[1,3^\ast]$, the orange arc is denoted $[2,4^\ast]$ and the red arc is denoted $[5,7^\ast]$. The arcs $[1,3^\ast]$ and $[5,7^\ast]$ are the same and correspond to the arc in the first picture of the twice-punctured disk of Figure \ref{examples_arcs}. The arc $[2,4^\ast]$ is not the same as the two others, it corresponds to the arc in the second picture of the twice-punctured disk of Figure \ref{examples_arcs}.

The second cylinder of Figure \ref{examples_arcs} illustrates the third observation above. The blue arc is denoted $[1^\ast,4]$ and the red arc is denoted $[2,5^\ast]$. These two arcs are the same and correspond to the arc in the third picture of the twice-punctured disk of Figure \ref{examples_arcs}.
\end{example}
\begin{figure}[h]
\includegraphics[scale=.8]{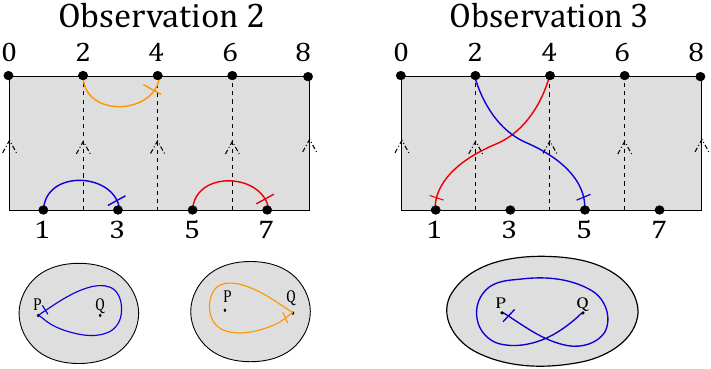} 
\caption{Example of equivalent curves}\label{examples_arcs}
\end{figure}

By the above observations, we can always write an arc $[x,y]$ as follows: 
\begin{enumerate}[label=\alph*.]
\item $x\in\{0,0^\ast,1,1^\ast \}$, $y\in\Z_{>0}\sqcup (\Z_{>0})^\ast$ and $x<y$; \label{conv_a}
\item if $\norm{x}\not\equiv \norm{y} \mod 2$ (i.e. when the arc connects $P$ and $Q$) then: \label{conv_b}
\begin{itemize}
\item if $\tagg(x)=\tagg(y)$, then we write the arc with $\norm{x}=0$;
\item if $\tagg(x)\neq \tagg(y)$, then we write the arc with $\norm{x}=1$.
\end{itemize}
\end{enumerate}

From now on, we will always use the conventions \ref{conv_a} and \ref{conv_b} to describe interior arcs in the twice-punctured disk $\D(n)$.

\subsection{Isomorphism of translation quivers}
Let $Q$ be a quiver, recall that an automorphism $\tau: Q\rightarrow Q$ is called a \textit{translation} if the predecessors of $v$ in $Q$ are equal to the successors of $\tau(v)$ in $Q$ for every vertices $v$. The quiver $Q$ with a translation $\tau$ is called a \textit{stable translation quiver}.

We construct a stable translation quiver with two connected components associated to $\D(n)$, called $\Gamma(\D(n))$.
The vertices of $\Gamma(\D(n))$ are the interior arcs of the twice-punctured disk $\D(n)$ and the arrows are defined as follow. 
Let $[x,y]$ and $[a,b]$ be interior arcs in $\D(n)$ satisfying the conventions \ref{conv_a} and \ref{conv_b}. Then there is an arrow from $[x,y]$ to $[a,b]$ if: 
\begin{itemize}
\item $\norm{x}=0$, $a=x$ and $b=y^\ast+1$ or if $a=x^\ast$ and $b=y-1$;
\item $\norm{x}=1$, $a=x$ and $b=y+1$ or if $a=x^\ast$ and $b=y^\ast-1$.
\end{itemize}
We define a translation map on $\Gamma(\D(n))$ by setting $\tau([x,y])=[x^\ast,y^\ast]$. This map is $2$-periodic.

\begin{prop} \label{Prop:translation}
$\Gamma(\D(n))$ is a stable translation quiver (in the sense of Riedtmann \cite{Riedtmann}) with two connected components and both components are tubes of rank $2$.
\end{prop}

\begin{proof}
We observe that two arcs $[x,y]$ and $[a,b]$ are connected by a sequence of arrows and inverse arrows if and only if $\norm{x}=\norm{a}$, so the quiver $\Gamma(\D(n))$ has two connected components. We will denote these two components as $\Gamma(\D(n))=\Gamma_0\sqcup \Gamma_1$, where $\Gamma_0$ contains the arcs of the form $[0,y]$ and $[0^\ast,y]$, and $\Gamma_1$ contains the arcs of the form $[1,y]$ and $[1^\ast,y]$. 
We observe that $\tau([x,y])=[x^*,y^*]$ so the arcs $[x,y]$ and $\tau([x,y])$ are in the same connected component. 

In the tube $\Gamma_0$, the predecessors of $[x,y]$ are $[x,y^*-1]$ and $[x^*,y+1]$ and the successors of $\tau([x,y])=[x^*,y^*]$ are $[x^*,y+1]$ and $[x,y^*-1]$. So the predecessors of $[x,y]$ coincide with the successors of $\tau([x,y])$ for all $[x,y]$ in $\Gamma_0$.

In the tube $\Gamma_1$, the predecessors of $[x,y]$ are $[x,y-1]$ and $[x^*,y^*+1]$ and the successors of $\tau([x,y])=[x^*,y^*]$ are $[x^*,y^*+1]$ and $[x,y-1]$. So the predecessors of $[x,y]$ coincide with the successors of $\tau([x,y])$ for all $[x,y]$ in $\Gamma_1$.

In both of the connected components, $\tau$ clearly defines a structure of stable translation quiver. 

We observe that $\tau(\tau([x,y]))=[x,y]$ for all $[x,y]$, so both connected components are tubes of rank $2$.
\end{proof}

\begin{example} A part of the translation quiver $\Gamma(\D(n))$.
\[\Gamma_0: \begin{tikzcd}[column sep=small]
	{[0^\ast,3^\ast]} && {[0,3]} && {[0^\ast,3^\ast]} \\
	& {[0,2^\ast]} && {[0^\ast,2]} && {[0,2^\ast]} \\
	{[0,1]} && {[0^\ast,1^\ast]} && {[0,1]}
	\arrow[dashed, no head, from=1-1, to=1-3]
	\arrow[from=1-1, to=2-2]
	\arrow[dashed, no head, from=1-3, to=1-5]
	\arrow[from=1-3, to=2-4]
	\arrow[from=1-5, to=2-6]
	\arrow[from=2-2, to=1-3]
	\arrow[dashed, no head, from=2-2, to=2-4]
	\arrow[from=2-2, to=3-3]
	\arrow[from=2-4, to=1-5]
	\arrow[dashed, no head, from=2-4, to=2-6]
	\arrow[from=2-4, to=3-5]
	\arrow[from=3-1, to=2-2]
	\arrow[dashed, no head, from=3-1, to=3-3]
	\arrow[from=3-3, to=2-4]
	\arrow[dashed, no head, from=3-3, to=3-5]
	\arrow[from=3-5, to=2-6]
\end{tikzcd}\]
\[\Gamma_1: \begin{tikzcd}[column sep=small]
	{[1,4^\ast]} && {[1^\ast,4]} && {[1,4^\ast]} \\
	& {[1^\ast,3]} && {[1,3^\ast]} && {[1^\ast,3]} \\
	{[1^\ast,2]} && {[1,2^\ast]} && {[1^\ast,2]}
	\arrow[dashed, no head, from=1-1, to=1-3]
	\arrow[from=1-1, to=2-2]
	\arrow[dashed, no head, from=1-3, to=1-5]
	\arrow[from=1-3, to=2-4]
	\arrow[from=1-5, to=2-6]
	\arrow[from=2-2, to=1-3]
	\arrow[dashed, no head, from=2-2, to=2-4]
	\arrow[from=2-2, to=3-3]
	\arrow[from=2-4, to=1-5]
	\arrow[dashed, no head, from=2-4, to=2-6]
	\arrow[from=2-4, to=3-5]
	\arrow[from=3-1, to=2-2]
	\arrow[dashed, no head, from=3-1, to=3-3]
	\arrow[from=3-3, to=2-4]
	\arrow[dashed, no head, from=3-3, to=3-5]
	\arrow[from=3-5, to=2-6]
\end{tikzcd}\]
\end{example}



From the previous observations and from Proposition \ref{Prop:translation}, we obtain the following isomorphism that will give us the geometric model for these two tubes:

\begin{theorem}
The translation quiver $\Gamma(\D(n))=\Gamma_0\cup\Gamma_1$ is isomorphic to the disjoint union of the two tubes $\T_1\sqcup\T_2$ of rank two of the Auslander-Reiten quiver of the cluster category of the algebra $\Lambda$ of affine type $D_n$.

\end{theorem}

\bibliographystyle{amsalpha}
\bibliography{biblio}
\end{document}